\documentclass{amsart}  

\usepackage{latexsym}
\usepackage{amsmath,amsthm}
\usepackage{amssymb}
\usepackage{graphicx}

\input diagxy
 \xyoption{curve}

\theoremstyle{plain}
\newtheorem{theorem}{Theorem}[]
\newtheorem{lemma} [theorem]{Lemma}
\newtheorem{proposition} [theorem]{Proposition}
\newtheorem{corollary} [theorem]{Corollary}

\theoremstyle{definition}
\newtheorem{definition}{Definition}

\theoremstyle{remark}

\newtheorem{rem}{Remark}
\newtheorem{Example}{Example}

\begin{document}

\author{Layth M. Alabdulsada }

\address{Layth M. Alabdulsada,  Institute of Mathematics, University of Debrecen, H-4002 Debrecen, P.O. Box 400, Hungary}
\email{layth.muhsin@science.unideb.hu}

\title{On the class of weakly almost contra-$T^*$-continuous functions}


\bibliographystyle{alpha}

\begin{abstract}
The aim of this paper is to introduce and investigate a new class of functions called weakly almost contra-$T^*$-continuity which is defined as a
function from an operator topological space $(X, \tau, T)$ into an arbitrary topological space $(Y, \delta)$. Furthermore, some new characterizations, several basic propositions are
proved and some
relevant counterexamples are provided.   \end{abstract}
\maketitle

\section{Introduction}
In the literature, a number of generalizations of open sets and its continuous functions have been considered. Indeed, many mathematicians worked in this area and made great contributions to develop several types of almost contra-continuous and weakly almost contra-continuous functions. These functions which are defined between two an arbitrary topological spaces have been discussed extensively in the literature.  For general reference, we refer the reader to J. Dontchev \cite{D1} in 1996, J. Dontchev and T. Noiri \cite{D2} in 1999, M. Caldas and S. Jafari, \cite{C1a} in 2001 and E. Ekici \cite{E1} in 2004.
C. W. Baker studied and developed several types of weakly contra-continuous functions (see for instance \cite{B}, \cite{B1}).  Moreover, many of the related concepts studied well such that this subject has been received much attention in the last decade. Among others, see \cite{E1}, \cite{E2}, \cite{E3}, \cite{E4} and \cite{E5}.

H. J. Mustafa et al. used a different technique to define the continuity of functions from an operator topological space $(X, \tau, T)$, being a topological space with an operator $T$ associated with the topology $\tau$, into an arbitrary topological space $(Y, \delta)$, we refer the reader to \cite{M2} \cite{M3}. Using the concept of $T^*$-open set in \cite{M3} they introduced and studied almost contra-$T^*$-continuous functions, several properties and characterizations of these functions are considered. In this paper, we continue this line to explore a new approach to weakly almost contra-continuity such that our goal is to introduce some definitions and investigate various properties of a new category of functions called weakly almost contra-$T^*$-continuous in topological spaces via utilizing the concept of $T^*$-open set.

In the sequel, we will present a number of concepts which are linked to our investigations. First, in Section 2 we give the basic definitions and notations. Afterward, in Section 3, we will pay our attention to discuss weakly almost contra-$T^*$-continuous functions and its relationships to several other close concepts. The following are the main results of this paper:
\begin{itemize}
    \item [(1)] Let $f: (X, \tau, T) \to (Y, \delta)$ be a function where $T(S) = \mathrm{Int(Cl}(S))$. Then $f$ is weakly almost contra-$T^*$-continuous if and only if, whenever $S$ is regular closed in $Y$, $V$ is regular open subset of $Y$, and $S \subseteq V$, then $\mathrm{Cl(Int}(f^{-1}(S)))\subseteq f^{-1}(V)$.
   \item [(2)] If $f: (X, \tau, T) \to (Y, \delta)$ is an almost-$T^*$-continuous function, then $f$ is weakly almost contra-$T^*$-continuous.
   \item [(3)] If $f: (X, \tau, T) \to (Y, \delta)$ is a weakly almost contra-$T^*$-continuous function, where $(Y, \delta)$ is an extremely disconnected space, then $f$ is almost-$T^*$-continuous.
   \item [(4)] Let $f$  be a function from an operator topological space $(X, \tau, T)$ into an extremely disconnected space $(Y, \delta)$. Then the weakly almost contra-$T^*$-continuity is equivalent to the almost-$T^*$-continuity.
   \item [(5)] If $f: (X, \tau, T) \to (Y, \delta)$ is an almost contra-$T^*$-continuous function, then $f$ is weakly almost contra-$T^*$-continuous.
   \item [(6)] Suppose that $f: (X, \tau, T) \to (Y, \delta)$ is a weakly almost contra-$T^*$-continuous function, then $f$ is slightly contra-$T^*$-continuous.
\item [(7)] Let $f : (X, \tau, T) \,\to \, (Y, \delta)$  is a weakly almost contra-$T^*$-continuous surjection and let $Y$ be a $\Sigma$-space. If $X$ is contra-$T^*$-compact then $Y$ is $R$-compact.
\item [(8)] If the function $f : (X, \tau, T) \,\to \, (Y, \delta)$ is weakly almost contra-$T^*$-continuous and $Y$ is Urysohn, then $G(f)$ has a $T^*$-regular and contra $T^*$-regular graph as well.
\item [(9)] Let $f : (X, \tau, T) \to (Y, \delta)$ be weakly almost contra-$T^*$-continuous and the images of g$T^*$r-closed sets are regular closed, then $f$ is a$T^*$r-irresolute.
\item [(10)] If $f : (X, \tau, T) \to (Y, \delta)$ is almost g$T^*$r-continuous and a$T^*$r-irresolute, then $f$ is weakly almost contra-$T^*$-continuous.
\end{itemize}

\section{Preliminaries}

In this section, we will present the definitions and the basic concepts that play an important role in this paper. The closure of $S$ will be denoted by $\mathrm{Cl}(S)$. The interior of $S$ will be denoted by $\mathrm{Int}(S)$.

\begin{definition}
 A subset $S$ of a topological space $(X, \tau)$ is said to be:
  \begin{itemize}
    \item [$\bullet$] {\em regular open} if $S = \mathrm{Int( Cl}(S))$, {\em regular closed} if $S = \mathrm{Cl(Int}(S))$ \cite{ST}.
    \item [$\bullet$] {\em pre-open} if $S \subseteq \mathrm{Int( Cl}(S))$, the complement of pre-open is a {\em pre-closed} \cite{M1}.
    \item [$\bullet$] {\em semi-open} if $S \subseteq \mathrm{Cl( Int}(S))$, the complement of semi-open is a {\em semi-closed} \cite{L1}.
    \item [$\bullet$] {\em $\alpha$-open} if $S \subseteq \mathrm{Int( Cl( Int}(S)))$, the complement of $\alpha$-open is a {\em $\alpha$-closed} \cite{NJ}.
    \item [$\bullet$] {\em $\beta$-open} if $S \subseteq \mathrm{Cl( Int( Cl}(S)))$, the complement of $\beta$-open is a {\em $\beta$-closed} \cite{A1}.
  \end{itemize}

\end{definition}
The $\beta$-closure of a set $S$ denoted by $\beta\mathrm{Cl}(S)$,
is the intersection of all $\beta$-closed sets containing $S$. The $\beta$-interior of a set $S$ denoted by $\beta\mathrm{Int}(S)$, is the union of all $\beta$-open sets contained in $S$. The preclosure, preinterior, semiclosure and semiinterior of a set $S$ denoted by $\mathrm{pCl}(S), \mathrm{pInt}(S), \mathrm{sCl}(S)\ \text{and} \ \mathrm{sInt}(S)$, respectively, are defined analogously.
We say that $V$ is clopen subset if $V$ is both open and closed.
Furthermore, we have for any set $S$ that $$\mathrm{pCl}(S) = S \cup \mathrm{Cl(Int}(S)),$$ for more details see \cite{E1}.
\begin{definition}
  A function $f: (X, \tau) \to (Y, \delta)$ is said to be {\em continuous (pre-continuous \cite{M1}, semi-continuous \cite{L1}, $\alpha$-continuous \cite{N31a}, $\beta$-continuous \cite{A1}, resp}.) if $f^{-1}(V)$ is open (pre-open, semi-open, $\alpha$-open, $\beta$-open, resp.) in $X$ for each open subset $V \subseteq Y$.
\end{definition}

\begin{definition}
  A function $f: (X, \tau) \to (Y, \delta)$ is said to be {\em contra-continuous \cite{D1} (contra-pre-continuous \cite{J1}, contra-semi-continuous \cite{D2}, contra-$\alpha$-continuous \cite{J2}, contra-$\beta$-continuous \cite{C1a}, resp}.) if $f^{-1}(V)$ is closed (pre-closed, semi-closed, $\alpha$-closed, $\beta$-closed, resp.) in $X$ for each open subset $V \subseteq Y$.
\end{definition}

\begin{definition}
  A function $f: (X, \tau) \to (Y, \delta)$ is said to be {\em almost contra-continuous \cite{S1a}(almost contra-pre-continuous \cite{E1}, almost contra-semi-continuous \cite{I1}, almost contra-$\alpha$-continuous \cite{N311}, almost contra-$\beta$-continuous \cite{C1a}, resp.)} if $f^{-1}(V)$ is closed (pre-closed, semi-closed, $\alpha$-closed, $\beta$-closed, resp.) in $X$ for every regular open $V$ of $Y$.
\end{definition}

\begin{definition}
 A function $f: (X, \tau) \to (Y, \delta)$ is said to be {\em weakly contra-continuous \cite{B} (weakly contra-pre-continuous \cite{B}, weakly contra-$\beta$-continuous\cite{B1}, resp.)} provided that, whenever $S \subseteq V \subseteq Y$, $S$ is closed in $Y$, and $V$ is open in $Y$, then $\mathrm{Cl}f^{-1}(S)\subseteq f^{-1}(V)$ ($\mathrm{pCl}f^{-1}(S)\subseteq f^{-1}(V)$, $\mathrm{\beta Cl}f^{-1}(S)\subseteq f^{-1}(V)$, resp.) in $X$.
 \end{definition}

\begin{definition} \cite{M2}
 Let $(X, \tau)$ be a topological space  and $P(X)$ be the power set of $X$. A function $T: P(X) \to P(X)$ is said to be an operator associated with topology $\tau$ on $X$ if $U\subseteq T(U)$ for all $U \in \tau$ and the triple $(X, \tau, T)$ is called an {\em operator topological space}.
\end{definition}

\begin{Example}\item
\begin{itemize}
  \item [$\bullet$] If $T$ is the identity operator, i.e., $T(S)= S$, then the triple $(X, \tau, T)$ will reduces to $(X, \tau)$, thus the operator topological space is the ordinary
topological space.
  \item [$\bullet$] Let $(X, \tau)$ be any topological space and function $T: P(X) \to P(X)$ such that $T(S) := \mathrm{Int(Cl}(S))$ for any $S \subseteq X$.
  Notice that if $U$ is open in $X$, then $$U \subseteq \mathrm{Int(Cl}(U)) = T(U).$$ Consequently, $T$ is an operator associated with the topology $\tau$ on $X$ and the triple
$(X, \tau, T)$ is an operator topological space.
\end{itemize}

\end{Example}

\begin{definition} \cite{M3}
 Let $(X, \tau, T)$ be an operator topological space and $S \subseteq X $, then $S$ is said to be {\em $T^*$-open} if $S \subseteq T(S)$ (observe that $S$ not necessarily open).
 The complement of $T^*$-open is called {\em $T^*$-closed}.

\end{definition}
 \begin{rem}
 \item
   \begin{itemize}
    \item [$\bullet$] $T^*\mathrm{Cl}(S) := \cap \ \{ U \mid  U \ \text{is}\  T^*\text{-closed},\ U \supseteq S \}$,
    \item [$\bullet$] if $T(S) = \mathrm{Int(Cl}(S))$, where $S \subseteq X$ then $T^*$-open set is exactly the pre-open set and $T^*\mathrm{Cl}(S) \equiv \mathrm{pCl}(S)$,
    \item [$\bullet$] if $T(S) = \mathrm{Cl(Int}(S))$, where $S \subseteq X$ then $T^*$-open set is exactly the semi-open set and $T^*\mathrm{Cl}(S) \equiv \mathrm{sInt}(S)$.
  \end{itemize}
 \end{rem}

\begin{definition} A function $f: (X, \tau, T) \to (Y, \delta)$ is said to be  \cite{M3}:
  \begin{itemize}
    \item {\em $T^*$-continuous} if $f^{-1}(V)$ is $T^*$-open in $X$ for each open subset $V \in \delta$.
    \item {\em almost $T^*$-continuous} if $f^{-1}(V)$ is $T^*$-open in $X$ for every regular open subset $V$ of $Y$.
    \item {\em contra-$T^*$-continuous} if $f^{-1}(V)$ is $T^*$-closed in $X$ for every open subset $V$ of $Y$.
    \item {\em almost contra-$T^*$-continuous} if $f^{-1}(V)$ is $T^*$-closed in $X$ for every regular open subset $V$ of $Y$.
    \end{itemize}
\end{definition}

\begin{definition} A function $f: (X, \tau, T) \to (Y, \delta)$ is called:
\begin{itemize}
\item {\em slightly contra-$T^*$-continuous} if, for every $x \in X$ and every clopen subset $V$ of $Y$ containing
$f(x)$ there exists a $T^*$-closed subset $f^{-1}(V)$ of $X$ with $x \in f^{-1}(V)$ and $f(f^{-1}(V))\subseteqq V$.
    \item {\em weakly contra-$T^*$-continuous} if for any $S \subseteq V \subseteq Y$, $S$ closed, $V$ open in $Y$, we have $T^*$Cl$f^{-1}(S)\subseteq f^{-1}(V)$.

  \item {\em weakly almost contra-$T^*$-continuous} if for every regular open subset $V$ of $Y$ and every regular closed subset $S$ of $Y$ with $S \subseteq V$, we have $T^*$Cl$f^{-1}(S)\subseteq f^{-1}(V)$ in $X$.

\end{itemize}
 \end{definition}
 Some examples of weakly almost contra-$T^*$-continuous functions will be shown later.
\begin{definition}
Let $(X, \tau)$ be a topological space, then $X$ is said to be {\em extremely disconnected} \cite{A5} whenever the closures of open sets are open.
\end{definition}

\section{Weakly almost contra-$T^*$-continuous functions}

\begin{lemma}
  Let $f: (X, \tau, T) \to (Y, \delta)$ be a function where $T(S) = \mathrm{Int(Cl}(S))$. Then $f$ is weakly almost contra-$T^*$-continuous if and only if, whenever $S$ is regular closed in $Y$, $V$ is regular open subset of $Y$, and $S \subseteq V$, then $\mathrm{Cl(Int}(f^{-1}(S)))\subseteq f^{-1}(V)$.
\end{lemma}
\begin{proof}
  Since $T(S) = \mathrm{Int(Cl}(S))$, $T^*$-openness plays the same role as pre-openness and $T^*\mathrm{Cl}(S) \equiv \mathrm{pCl}(S)$. Turning on to our proof, we know that $$
   \mathrm{pCl}(S) = S \cup \mathrm{Cl(Int}(S)).
  $$
  But pCl($f^{-1}(S)) \subseteq f^{-1}(V)$, then $f^{-1}(S) \cup \mathrm{Int(Cl(}f^{-1}(S))) \subseteq f^{-1}(V)$, therefore $\mathrm{Int(Cl(}f^{-1}(S))) \subseteq f^{-1}(V)$.
\end{proof}

\begin{rem} \label{SV}
\item
\begin{itemize}

\item  Every contra-$T^*$-continuous function are automatically weakly contra-$T^*$-continuous function, since $S \subseteq V $ implies $f^{-1}(S) \subseteq f^{-1}(V)$ and $T^*$$f^{-1}(S) \subseteq$ $T^*$$f^{-1}(V).$ If $f^{-1}(V)$ is $T^*$-closed, then $T^*$$f^{-1}(V) \subseteq f^{-1}(V)$, so we conclude that $T^*$$f^{-1}(S) \subseteq f^{-1}(V)$.
    \item  Following the same technique as above one can check that every weakly contra-$T^*$-continuous function is a weakly almost contra-$T^*$-continuous function.

\end{itemize}
\end{rem}

\begin{proposition}\label{1}
  If $f: (X, \tau, T) \to (Y, \delta)$ is an almost-$T^*$-continuous function, then $f$ is weakly almost contra-$T^*$-continuous.
\end{proposition}
\begin{proof}
   Suppose that $f$ is an almost-$T^*$-continuous function. First, fix that $S \subseteq V \subseteq Y$ such that $S$ is regular closed in $Y$ and $V$ is regular open in $Y$. Now, $f^{-1}(S)$ is $T^*$-closed, under the hypothesis that $f$ is an almost-$T^*$-continuous function and thus, $T^*$Cl$f^{-1}(S)\subseteq f^{-1}(S)\subseteq f^{-1}(V).$ Consequently, $f$ is a weakly almost contra-$T^*$-continuous function.
\end{proof}

 \begin{proposition}\label{2}
 If $f: (X, \tau, T) \to (Y, \delta)$ is a weakly almost contra-$T^*$-continuous function, where $(Y, \delta)$ is an extremely disconnected space, then $f$ is almost-$T^*$-continuous.
\end{proposition}
\begin{proof}
  Let $V$ be a regular closed subset of $Y$. Under the conditions stated above that $(Y, \delta)$ is extremely disconnected, $V$ is clopen and hence $V$ is also regular open. Therefore, $T^*$Cl$f^{-1}(V)\subseteq f^{-1}(V)$. By assumption $f$ is weakly almost contra-$T^*$-continuous functions, from what we conclude that $f^{-1}(V)$ is $T^*$-closed. Thus, $f$ is an almost-$T^*$-continuous function.
  \end{proof}

  One can prove immediately the next corollary from Propositions \ref{1} and \ref{2}.

  \begin{corollary}
   Let $f$  be a function from an operator topological space $(X, \tau, T)$ into an extremely disconnected space $(Y, \delta)$. Then the weakly almost contra-$T^*$-continuity is equivalent to the almost-$T^*$-continuity.
  \end{corollary}

\begin{proposition}
  Let $f: (X, \tau, T) \to (Y, \delta)$ be an almost contra-$T^*$-continuous function, then $f$ is weakly almost contra-$T^*$-continuous.
\end{proposition}
\begin{proof}
  Assume that $f: (X, \tau, T) \to (Y, \delta)$ is an almost contra-$T^*$-continuous function. Let $S$ is a regular closed in $Y$ and $V$ is a regular open in $Y$ such that $S \subseteq V \subseteq Y$, since $f$ satisfies the property of almost contra-$T^*$-continuous. Therefore, $f^{-1}(V)$ is $T^*$-closed and therefore, $T^*$Cl$f^{-1}(S)\subseteq$ $T^*$Cl$f^{-1}(V)\subseteq f^{-1}(V).$ We thus obtain $f$ is a weakly almost contra-$T^*$-continuous function.
\end{proof}

\begin{proposition}
  Suppose that $f: (X, \tau, T) \to (Y, \delta)$ is a weakly almost contra-$T^*$-continuous function, then $f$ is slightly contra-$T^*$-continuous.
\end{proposition}
\begin{proof}
  We consider $f: (X, \tau, T) \to (Y, \delta)$ to be weakly almost contra-$T^*$-continuous and $V$ be a regular clopen (i.e., $V$ is regular open and regular closed) subset of $Y$. Then, since $V \subseteq V \subseteq Y$. This implies that $T^*$Cl$f^{-1}(V)\subseteq f^{-1}(V)$. Therefore, $f^{-1}(V)$ is $T^*$-closed and $f$ is a slightly contra-$T^*$-continuous function, as wanted to be shown.\end{proof}

  \begin{corollary}
  If $f: (X, \tau, T) \to (Y, \delta)$ is a weakly contra-$T^*$-continuous function, then $f$ is slightly contra-$T^*$-continuous.
\end{corollary}

Consequently, from what we have already proved, one can consider the following diagram: (C.= continuous)
\\

{\centering
 weakly contra-C.  \\
 $\Downarrow$ \\
weakly contra-$T^*$-C.\\
$\Downarrow$ \\
almost-$T^*$-C. $\Longrightarrow$ weakly almost contra-$T^*$-C. $\Longrightarrow$ slightly contra-$T^*$-C. \\
$\Uparrow$ \\
 \ \ \ \ \ \  \ \ \ \ \ \ \ \ \ \ \  \ \ \ \ \ \ \ \ \ \ \ \ \ \ \ \ \ \ \ \ \ \ \ \ almost contra-$T^*$-C.
}
\newline

The next examples show that, in general, none of the above implications are reversible.
\begin{Example}
\item
  \begin{itemize}

    \item [$\bullet$] Let $f: (\mathbb{R}, \tau, T) \,\to \, (\mathbb{R}, \delta)$ be the identity function such that $\tau =\{\mathbb{R}, \emptyset, \{0\}\}$ and $T:P(\mathbb{R}) \to P(\mathbb{R})$ is defined by $T(S)= \mathrm{Int(Cl(}S))$ and $\delta$ is the usual topology on $\mathbb{R}$. Since $(\mathbb{R}, \delta)$ is connected, $f$ is slightly contra-$T^*$-continuous function. However, $f$ is not weakly almost contra-$T^*$-continuous.
        To check this one can consider $S=[0, 1]$ and $V=(-2, 2)$, then $S$ is regular closed in $(\mathbb{R}, \delta)$ and $V$ is regular open in $(\mathbb{R}, \delta)$ with $S \subseteq V$, but $T^*$Cl$(f^{-1}(S)) = \mathrm{pCl}(f^{-1}(S)) \nsubseteq f^{-1}(V).$
    \item [$\bullet$] Let us consider the same identity function from $(X, \tau, T)$ into $(X, \delta)$, assume that $X=\{a, b, c \}$ have the following topologies $\tau = \{X, \emptyset , \{a\} \}$, and $\delta= \{X, \emptyset, \{c\}, \{a, c\}, \{b, c\}\}$ and $T:P(X) \to P(X) $  defined by $T(S)= \mathrm{Int(Cl}(S))$. Hence, the only regular open sets in $(X, \delta)$ are $X$ and $\emptyset$ where $f$ is weakly almost contra-$T^*$-continuous. Moreover, $f$ is not weakly contra-$T^*$-continuous. Since for $S=\{a\}$ and $V= \{a, c\}$ in the same space $(X, \delta)$, indeed $T^*$Cl$(f^{-1}(S)) = \mathrm{pCl}(f^{-1}(S)) \nsubseteq f^{-1}(V).$
    \item [$\bullet$]Suppose that $f: (X, \tau, T) \to (X, \delta)$ is the identity function, where $X=\{a, b, c\}$ and its topology given by $\tau = \{X, \emptyset , \{a\}, \{b\}, \{a, b\}\}$ such that $T:P(X) \to P(X)$ defined as following $T(S)= \mathrm{Int(Cl}(S))$. Then $f$ is almost $T^*$-continuous but not almost contra-$T^*$-continuous. Note that $S=\{a\}$ is regular open in $X$ and $f^{-1}(S)$ not $T^*$-closed.
    \item [$\bullet$]Let $f: (X, \tau, T) \,\to \, (X, \delta)$ be the identity function, defined by $X=\{a, b, c \}$ where its topologies are described by
    $\tau= \{X, \emptyset, \{c\}\} $ and $\delta= \{X, \emptyset, \{a\}, \{b\}, \{a, b\}\}.$
    As above, $T:P(X) \to P(X)$ is given by the following form $T(S)= \mathrm{Int(Cl}(S))$. Under the above assumptions, $f$ is an almost contra-$T^*$-continuous function. Consequently,  $S=\{a\}$ is regular open in $(X, \delta)$, but $f^{-1}(S)$ is not $T^*$-open.
  \end{itemize}
\end{Example}

\begin{definition} Let $(X, \tau)$ be a topological space, then $X$  is said to be:
\begin{itemize}
  \item [$\bullet$]  {\em $\Sigma$-space} \cite{M11} provided that every open set is the union of regular closed sets.
  \item [$\bullet$]  {\em $R$-compact} \cite{SM} if every regular open cover of $X$ has a finite subcover.
\end{itemize}
\end{definition}
Taking into account the operator topological space $(X, \tau, T)$ we define the following:
\begin{definition}
  Let  $f: (X, \tau, T)$ be an operator topological space, then $X$  is said to be {\em contra $T^*$-compact} if every cover of $X$ by $T^*$-closed sets has a finite subcover.
\end{definition}

\begin{proposition}
 Let $f : (X, \tau, T) \,\to \, (Y, \delta)$  is a weakly almost contra-$T^*$-continuous surjection and let $Y$ be a $\Sigma$-space. If $X$ is contra-$T^*$-compact then $Y$ is $R$-compact.
\end{proposition}

\begin{proof}
Let $C$ be a cover of $Y$ by regular open sets. Let $y \in Y$ and let $V_y \in C$ such that $y \in V_y$. Since $Y$ is a $\Sigma$-space, there exists a regular closed set $S_y$ such that $y \in S_y \in V_y$. Since $f$ is weakly almost contra-$T^*$-continuous, $T^*$Cl$(f^{−1}(S_y)) \subseteq f^{-1}(V_y)$. It follows that \{$T^*$Cl$(f^{-1}(S_y))| y \in Y \}$ is a cover of $X$ by $T^*$-closed sets. Since $X$ is contra $T^*$-compact, there exists a finite subcover \{$T^*$Cl$(f^{-1}(V_{y_i}))| i = 1, 2, ..., n\}$ It then follows that $$X=\cup^n_{i=1} (f^{-1}(V_{y_i})),$$ $$Y= f(X) = f(\cup^n_{i=1} (f^{-1}(V_{y_i})))=\cup^n_{i=1} (V_{y_i}),$$ this shows that $Y$ is $R$-compact.

\end{proof}
\begin{definition}
  A topological space $(X, \tau)$  is said to be {\em Urysohn} \cite{A6}, if for every pair of distinct points $x$ and $y$ in $X$, there exist open sets $U$ and $V$ such that $x \in U, \ y \in V$ and $\mathrm{Cl}(U) \cap \mathrm{Cl}(V) = \emptyset$.
\end{definition}

\begin{definition} Let $f: (X, \tau, T) \to (Y, \delta)$ be given.
\begin{itemize}
  \item [$\bullet$]  The graph $G(f)$ of a function $f$ is said to be {\em $T^*$-regular} whenever $(x, y) \in X \times Y\setminus G(f)$ there exist a $T^*$-closed set $U$ in $X$ containing $x$ and a regular open set $V$ in $Y$ containing $y$ such that $(U \times V ) \cap G(f) =\emptyset$. This is equivalent to $f(U) \cap V =\emptyset. $
  \item [$\bullet$] $f$ has a {\em contra $T^*$-regular graph} under the condition that for every $(x, y) \in X \times Y \setminus G(f)$ there exist a $T^*$-closed set $U$ in $X$ containing $x$ and a regular closed set $V$ in $Y$ containing $y$ such that $(U \times V ) \cap G(f) =\emptyset$.
  \item [$\bullet$] For any an operator topological space $(X, \tau, T)$  and $S \subseteq X$, we call $S$ {\em generalized $T^*$-regular closed} (briefly g$T^*$r-closed) if it is satisfying $T^*\mathrm{Cl}(S) \subseteq U$ whenever $S \subseteq U$ and $U$ is regular open.
  \item [$\bullet$] $f$ is called {\em approximately $T^*$-regular irresolute} (briefly a$T^*$r-irresolute) if $T^*\mathrm{Cl}(S) \subseteq f^{−1}(V)$ whenever $V$ is regular open, $S$ is g$T^*$r-closed, and $S \subseteq f^{−1}(V).$
  \item [$\bullet$] $f $ is said to be an {\em almost g$T^*$r-continuous function} whenever $f^{-1}(S)$ is $gT^*r$-closed for every regular closed subset $S$ of $Y$.
\end{itemize}
\end{definition}

\begin{proposition}
If the function $f : (X, \tau, T) \,\to \, (Y, \delta)$ is weakly almost contra-$T^*$-continuous and $Y$ is Urysohn, then $G(f)$ has a $T^*$-regular and contra $T^*$-regular graph as well.
\end{proposition}

\begin{proof}
Let $(x, y) \in X \times Y\setminus G(f)$. Then, since $y \neq  f(x)$ and $Y$ is Urysohn, there exist open sets $V$ and $W$ in $Y$ such that $y \in V$ and $f(x) \in W$ and $\mathrm{Cl}(V) \cap \mathrm{Cl}(W) = \emptyset$. Then we see that $\mathrm{Cl}(W) \subseteq Y \setminus \mathrm{Cl}(V)$, \ $\mathrm{Cl}(W)$ is regular closed, and $Y \setminus \mathrm{Cl}(V)$ is regular open. Since $f$ is weakly almost contra-$T^*$-continuous, $$T^*\mathrm{Cl}(f^{-1}(\mathrm{Cl}(W))) \subseteq f^{-1}(Y \setminus \mathrm{Cl}(V)).$$ It then follows that $(x, y) \in T^*\mathrm{Cl}(f^{-1}(\mathrm{Cl}(W))) \times \mathrm{Int(Cl}(V))$.

Let $U= T^*\mathrm{Cl}(f^{-1} (\mathrm{Cl}(W)))$, $U$ is $T^*$-closed. Since $\mathrm{Int(Cl}(V))$ is regular open, $$(U \times \mathrm{Int(Cl}(V))) \cap G(f),$$ which proves that $G(f)$ is $T^*$-regular.

To prove the second property,
let $(x, y) \in X \times Y\setminus G(f)$. As an above $y \neq  f(x)$ and $Y$ is Urysohn, there exist open sets $V$ and $W$ in $Y$ such that $y \in V$ and $f(x) \in W$ and $\mathrm{Cl}(V) \cap \mathrm{Cl}(W) = \emptyset$. Therefore $\mathrm{Cl}(V) \subseteq Y \setminus \mathrm{Cl}(W)$, since $\mathrm{Cl}(W)$ is regular closed, and $Y \setminus \mathrm{Cl}(W)$ is regular open. Moreover,  $f$ is weakly almost contra-$T^*$-continuous, we obtain $$T^*\mathrm{Cl}(f^{-1}(\mathrm{Cl}(V))) \subseteq f^{-1}(Y \setminus \mathrm{Cl}(W)).$$ It then follows that $$(x, y) \in (X \setminus T^*\mathrm{Cl}(f^{-1}(\mathrm{Cl}(V)))) \subseteq X \times Y \setminus G(f)$$ and we have that $G(f)$ is a contra $T^*$-regular graph.

\end{proof}

\begin{proposition}
 Let $f : (X, \tau, T) \to (Y, \delta)$ be weakly almost contra-$T^*$-continuous and the images of g$T^*$r-closed sets are regular closed, then $f$ is a$T^*$r-irresolute.
\end{proposition}
\begin{proof}
  Let $V$ be a regular open subset of $Y$ and let $S$ be a g$T^*$r-closed subset of $X$ such that $S \subseteq f^{-1}(V)$. Then $f(S)$ is regular closed and $f(S) \subseteq V$. Since $f$ is weakly almost contra-$T^*$-continuous, $T^*$Cl($f^{-1}(f(S)))\subseteq f^{-1}(V)$. Therefore $T^*$Cl(S)$\subseteq f^{-1}(V)$ and hence $f$ is $aT^*r$-irresolute.
\end{proof}

\begin{proposition}
  If $f : (X, \tau, T) \to (Y, \delta)$ is almost g$T^*$r-continuous and a$T^*$r-irresolute, then $f$ is weakly almost contra-$T^*$-continuous.
\end{proposition}
\begin{proof}
  Assume $S \subseteq V \subseteq Y$, where $S$ is regular closed in $Y$ and $V$ is regular open in $Y$. Since $f$ is almost g$T^*$r-continuous, $f^{-1}(S)$ is g$T^*$r-closed. Then, since $f^{-1}(S) \subseteq f^{-1}(V)$ and $f$ is a$T^*$r-irresolute, $T^*$Cl($f −1(f(A)))\subseteq f^{-1}(V)$, which proves that $f$ is weakly almost contra-$T^*$-continuous.
\end{proof}

\section*{Acknowledgement}

I offer my sincerest gratitude to my supervisor Dr. L\'{a}szl\'{o} Kozma,  for carefully reviewing the work, providing useful suggestions.


\begin{thebibliography}{99}

\bibitem{A1}
M.E. Abd El-Monsef, S.N. El-Deeb and R.A. Mahmoud.
$\beta$-open sets and $\beta$-continuous mapping. {\em Bull. Fac. Sci. Assiut Univ. A}, {\bf 12} (1983), no. 1, 77--90.

\bibitem{A5}
 A.V. Arkhangelskii and V.I. Ponomarev,
Fundamentals of general topology: problems and exercises,  {\em Państwowe Wydawnictwo Naukowe (PWN), Warsaw,} (1984) (Translated from Russian)
\bibitem{A6}
 S. P. Arya and M. P. Bhamini
Some generalizations of pairwise Urysohn spaces,
{\em Indian J. Pure Appl. Math.} {\bf 18} (1987), no. 12, 1088--1093.
\bibitem{B}
 C. W. Baker,
 Weakly contra-continuous functions, {\em Int. J. Pure Appl.
Math.,} {\bf 40} (2007), no. 2, 265--271.

\bibitem{B1}
 C. W. Baker,
 Weakly contra-$\beta$-continuous functions and strongly $S\beta$-closed sets, {\em J. Pure Math.,} {\bf 24} (2007),  31--38.


\bibitem{C1a}
M. Caldas and S. Jafari,
Some properties of contra-$\beta$-continuous functions, {\em Mem. Fac. Sci.
Kochi Univ. Ser. A Math.,} {\bf 22} (2001),  19--28.

\bibitem{D1}
J. Dontchev,
Contra-continuous functions and strongly $S$-closed spaces, {\em Internat.J.
Math. Math. Sci.,} {\bf 19} (1996), no. 2, 303--310.

\bibitem{D2}
J. Dontchev and T. Noiri,
Contra-semicontinuous functions, {\em Math. Pannon.,} {\bf 10} (1999), no. 2,  159--168.

\bibitem{E1}
E. Ekici,
Almost contra-precontinuous functions, {\em Bull. Malaysian Math. Sc. Soc.,}  {\bf 27(2)} (2004), no. 1, 53--65.

\bibitem{E2}
E. Ekici,
($\delta$-pre, $s$)-continuous functions, {\em Bull. Malaysian Math. Sc. Soc.,}  {\bf 27(2)} (2004), no. 2, 237--251.

\bibitem{E3}
E. Ekici,
On the notion of ($\gamma, s$)-continuous functions, {\em Demonstratio Math.,}  {\bf 38} (2005), no. 3, 715--727.

\bibitem{E4}
E. Ekici,
On contra $\pi$g-continuous functions, {\em Chaos Solitons Fractals,}  {\bf 35} (2008), no. 1, 71--81.

\bibitem{E5}
E. Ekici,
New forms of contra-continuity, {\em Carpathian J. Math.,}  {\bf 24} (2008), no. 1, 37--45.


\bibitem{I1}
D. Iyappan and N. Nagaveni,
 The separation axioms on semi generalized $b$-closed sets., {\em Int. J. Math. Sci. Eng. Appl.,}  {\bf 4} (2010), no. 2, 149--159.


\bibitem{J1}
S. Jafari and T. Noiri,
On contra-precontinuous functions, {\em Bull. Malaysian Math. (2),}  {\bf 25} (2002), no. 2, 15--128.
\bibitem{J2}
S. Jafari and T. Noiri,
Contra-$\alpha$-continuous functions between topological spaces, {\em Iran. Int.
J. Sci.,}  {\bf 2} (2001), no. 2, 153--1167.
\bibitem{L1}
N. Levine,
Semi-open sets and semi-continuity in topological spaces, {\em Amer. Math. Monthly,} {\bf 70} (1963),
36--41.

\bibitem{M1}
A.S. Mashhour, M.E. Abd El-Monsef and S.N. El-Deeb,
On precontinuous and weak precontinuous mappings,
{\em  Proc. Math. Phys. Soc. Egypt,} {\bf 53} (1982), 47--53.

\bibitem{M11}
A. Miller,
Special subsets of the real line,
{\em Handbook of Set-Theoretic Topology
(K. Kunen and J. E. Vaughan, Eds.} North-Holland, Amsterdam (1984), 201--235.


\bibitem{M2}

H. J. Mustafa, A. L. Moussa, A. K. Mazal,
Operator topological spaces, {\em Journal of the College of Education},
Almustansiriyah Univ.{\bf (1)} (2011), 213--221.


\bibitem{M3}
H. J. Mustafa and L. M. Alabdulsada,
On almost contra $T^*$-continuous functions, {\em  J. of Kufa for Math. and Comp.,}  {\bf1} (2012), no.6,  1--6.


\bibitem{NJ}
O. Nj{\aa}stad,
On some classes of nearly open sets, {\em  Pacific J. Math.,}  {\bf 15} (1965), 961--970.

\bibitem{N311}
T. Noiri,
Almost-$\alpha$-continuous functions, {\em  Kyungpook Math. J.,}  {\bf 28} (1988), no. 1, 71--77.

\bibitem{N31a}
T. Noiri,
On $\alpha$-continuous functions, {\em \v{C}asopis P\v{e}st. Mat.,}  {\bf 109} (1984), no. 2, 118--126.


\bibitem{S1a}
M.K. Singal and A.R. Singal,
Almost continuous mappings, {\em Yokohama
Math. J.,} {\bf 16} (1968), 63--73.

\bibitem{SM}
M.K. Singal and A. Mathur,
On nearly-compact spaces, {\em Boll. Un. Mat. Ital. (4),} {\bf 2} (1969),  702--710.

\bibitem{ST}
M.H. Stone,
Applications of the theory of Boolean rings to general topology, {\em Trans. Amer. Math. Soc.,} {\bf 41} (1937), no. 3, 375--481.
\end{thebibliography}
\end{document}